\documentclass{amsart}
\usepackage{amsmath,amsfonts,amssymb}
\usepackage{amsthm}
\usepackage{ifthen}
\usepackage{longtable}
\usepackage{array}
\usepackage{url}
\usepackage{hyperref}

\newcommand{\Z}{\mathbb{Z}}
\newcommand{\Q}{\mathbb{Q}}
\newcommand{\R}{\mathbb{R}}
\newcommand{\C}{\mathbb{C}}

\newcommand{\ord}{\mathcal{O}}
\newcommand{\order}{\mathrm{ord}}

\renewcommand{\P}{\mathbb{P}}

\DeclareMathOperator{\Gal}{Gal} 

\DeclareMathOperator{\Hom}{Hom}

\newtheorem*{theorem*}{Theorem}
\newtheorem{theorem}{Theorem}
\newtheorem{lemma}{Lemma}
\newtheorem{corollary}{Corollary}
\newtheorem{proposition}{Proposition}

\theoremstyle{remark}

\begin{document}
\title[$S$-unit equations]{On the number of solutions to $S$-unit equations}

\author[V. Ziegler]{Volker Ziegler}
\address{V. Ziegler,
University of Salzburg,
Hellbrunnerstrasse 34,
A-5020 Salzburg, Austria}
\email{volker.ziegler\char'100plus.ac.at}

\begin{abstract}
Let $K$ be a number field and $S$ a finite set of places including the archimedian ones. Let $(a_1,a_2,a_3)\in K^*$ we consider the $S$-unit equation $a_1x+a_2y=a_3$ with $x,y\in U_S$. In this paper we prove that for fixed $a_1,a_2$ the equation has at most $\frac{|S|-1}2$ solutions provided that the height of $a_3$ is sufficiently large. We want to emphasize that our results are effective.
\end{abstract}

\maketitle

\section{Introduction}

Let $K$ be a number field and $S$ a finite set of places including the archimedian ones. Let us denote by $U_S$ the set of $S$-units in $K$. Then for a fixed triple $(a_1,a_2,a_3)\in (K^*)^3$ we consider the $S$-unit equation
\begin{equation}\label{eq:SUeq}
 a_1x+a_2y=a_3,\qquad x,y\in U_S.
\end{equation}
More than a century ago Siegel \cite{Siegel:1921} proved that if $S$ consists of the archimedian primes, then Equation \eqref{eq:SUeq} has at most finitely many solutions $x,y\in \ord_K^*$ where $\ord_K$ is the ring of algebraic integers of $K$. Finiteness for the general case was proved by Lang \cite{Lang:1960} and effective results where subsequently proved by Gy\H{o}ry \cite{Gyoery:1974,Gyoery:1979}. For more on the history of unit equations and their results we refer to \cite{Evertse-Gyoery:UnitEq}.

Although it is known that the number of solutions to \eqref{eq:SUeq} can be large (e.g. see \cite{Erdos:1988}), in many cases one can prove that only few solutions exist, if not all of $a_1,a_2$ and $a_3$ are small. Let us make this statement more precise.

We call the triples $(a_1,a_2,a_3)$, $(a'_1,a'_2,a'_3)\in (K^*)^3$ $S$-equivalent, if there exist $\lambda\in K^*$, $S$-units $\epsilon_1,\epsilon_2, \epsilon_3\in U_S$ and a permutation $\sigma$ on the set $\{1,2,3\}$ such that
$$(\lambda \epsilon_1 a'_{\sigma(1)},\lambda \epsilon_2 a'_{\sigma(2)},\lambda \epsilon_3 a'_{\sigma(3)})=(a_1,a_2,a_3).$$
Note that if $(x,y)=(x_1,x_2)$ is a solution to $a_1x+a_2y=a_3$, then
$(x,y)=\left(\frac{x_{\sigma(1)}\epsilon_{\sigma(3)}}{x_{\sigma(3)}\epsilon_{\sigma(1)}},\frac{x_{\sigma(2)}\epsilon_{\sigma(3)}}{x_{\sigma(3)}\epsilon_{\sigma(2)}}\right)$ is a solution to $a'_1x+a'_2y=a'_3$, where we set $x_3=1$.

Denote by $\nu(a_1,a_2,a_3)$ the number of solutions to \eqref{eq:SUeq}. Then we have that $\nu(a_1,a_2,a_3)=\nu(a'_1,a'_2,a'_3)$, if $(a_1,a_2,a_3)$ and $(a'_1,a'_2,a'_3)$ are $S$-equivalent. Although it has been shown by Erd\H{o}s, Stewart and Tijdeman \cite{Erdos:1988} that an $S$-unit can have many solutions, it has been shown by Evertse, Gy\H{o}ry, Stewart and Tijdeman~\cite{Evertse:1988}
that most $S$-unit equations have only few solutions. For the rational case we refer to a result of Brindza and Gy\H{o}ry \cite{Brindza:1990} and a recent result due to Pint\'{e}r \cite{Pinter:2026}.

Let us have a closer look on the result due to Evertse, et.al. \cite{Evertse:1988}. They showed that if an $S$-unit Equation~\eqref{eq:SUeq} has at least $|S|+2$ solutions, then there exists a triple $(a'_1,a'_2,a'_3)$ equivalent to $(a_1,a_2,a_3)$ such that
$$\max\{h(a'_1),h(a'_2),h(a'_3)\}\leq C,$$
where $C$ is an effectively computable constant depending on $S$ and $K$. In this paper we prove a result similar to the results due to Evertse, Gy\H{o}ry, Stewart and Tijdeman~\cite{Evertse:1988}. In particular, the main result of this paper implies the following:

\begin{corollary}\label{cor:main}
  Assume that $\nu(a_1,a_2,a_3)\geq \frac{|S|+1}2$ and $a_1/a_2,a_1/a_3,a_2/a_3\not\in U_S$, then $(a_1,a_2,a_3)$ is $S$-equivalent to a triple $(a'_1,a'_2,a'_3)$, with $h(a'_3)\leq C$, where $C$ is an explicitly computable constant depending only on $a_1$, $a_2$, $S$ and $K$.
\end{corollary}

In the next section we will introduce several notations and discuss some results from algebraic number theory concerning $S$-units. This will allow us to formulate a more concise statement, than provided by Corollary \ref{cor:main}. We will also formulate a related result concerning the problem which rational integers can be presented as the sum of two units. In Section \ref{sec:DioApprox} we will present results due to Matveev \cite{Matveev:2000} and Yu \cite{Yu:2007} on linear forms in (complex and $p$-adic) logarithms. We apply these results to $S$-unit equations and prove two key results (Lemma \ref{lem:4terms} and Lemma \ref{lem:4terms-conj}) that are essential in proving our main results. Hence, Theorem \ref{th:number} will be proved in Section~\ref{sec:number} and our result concerning the representation of rational integers as a sum of two units (Theorem \ref{th:unit-rep}) will be proved in Section \ref{sec:units}.

\section{Notations and Results}

Let $K$ be a number field of degree $d$ and discriminant $D_K$. We denote by $P_K$ be the set of places of $K$ and by $P_\Q=\mathbb P\cup \{\infty\}$ the set of places of $\Q$. We assume that the corresponding absolute values $P_\Q$ of $\Q$ are normalized as follows. For $p\in \P$ we assume that $|p|_p=p^{-1}$ and $|\cdot|_\infty$ is the usual absolute value. For $v\in P_K$ we denote by $|\cdot|_v$ the absolute value such that $|x|_v=|x|_p$ if $v$ lies above $p\in P_\Q$ and $x\in \Q$. Moreover, we write $\|\cdot\|_v=|\cdot|_v^{d(v)}$, where $d(v)=[K_v:\Q_p]/[K:\Q]$, $\Q_p$ is the completion of $\Q$ at $p$, and $K_v$ is the completion of $K$ at $v$. Note that by this choice the product formula holds:
\begin{equation}\label{eq:ProductFormula}
 \prod_{v\in P_K} \|\alpha\|_v =1 \qquad \alpha \in K^*.
\end{equation}
We define the absolute logarithmic Weil height by
$$
h(\alpha)=\frac 1{[K:\Q]} \sum_{v\in P_K} \max\{0,\log\|\alpha\|_v\},
$$
where $K$ is some number field with $\alpha\in K$. Note that indeed the definition of $h(\alpha)$ is independent of the choice of the ambient number field $K$.

Let us assume that $v_{\mathfrak p}\in P_K$ is finite and corresponds to the prime ideal $\mathfrak p|p$, which lies above the prime number $p$. Then we denote by $\order_{\mathfrak p}(\alpha)$ the exponent of $\mathfrak p$ with which $\mathfrak p$ divides the principal ideal $(\alpha)$. This implies $\order_{\mathfrak p}(\alpha)=\frac{-\log \|\alpha\|_v}{\log p}$.

Let us note that for a fixed number field $K$ of degree $d=[K:\Q]$ and a given constant $c$ there are at most finitely many $\alpha\in K^*$ such that $h(\alpha)\leq c$. Also note that $h(\alpha)=0$ holds if and only if $\alpha$ is a root of unity. Thus for a fixed number field $K$ there exists a positive constant $\delta_K$ such that $h(\alpha)\geq \delta_K/d$ for all $\alpha\in K^*$, with $\alpha$ not a root of unity. Due to a result of Voutier \cite{Voutier:1996} we can choose $\delta_K=\lambda(d):=\frac{2}{(\log (3d))^3}$, if $d\geq 2$. Since $\lambda(1)\geq \log 2$
we can choose
\begin{equation}\label{eq:Voutier-est}
 \delta_K=\lambda(d):=\frac{8}{9(\log (3d))^3}
\end{equation}
for all $d\geq 1$.

Let $S\subseteq P_K$ be finite with $|S|=s$ and assume that $S$ contains all archimedian places. Then $\alpha\in K$ is called an $S$-unit if $\|\alpha\|_v=1$ for all $v\not \in S$ respectively an $S$-integer if $\|\alpha\|_v\geq 1$ for all $v\not \in S$. The set of all $S$-units is denoted by $U_S$ and the set of all $S$-integers is denoted by $\ord_S$. If $S$ consists of all archimedian places, then $U_S=U_K$ is the group of ordinary units and $\ord_S=\ord_K$ is the ring of algebraic integers. We also define the $S$-norm by
$$N_S(\alpha):=\prod_{v\in S} \|\alpha\|_v.$$
Let us note that in case that $S$ is the set of all archimedian places, then we have $N_S(\alpha)=|N_{K/\Q}(\alpha)|$.

By Dirichlet's unit theorem (see e.g. \cite[Corollary, page 105]{Lang:AlgZT}) we know that for a given set of places $S=\{v_1,\dots,v_s\}$ there exist units $\eta_1,\dots,\eta_{s-1}$ and a root of unity $\zeta$ such that $U_S=\langle \zeta,\eta_1,\dots,\eta_{s-1}\rangle$. Then the $S$-Regulator is defined as the absolute value of one of the $s-1$ minors of the matrix $R=(\log \|\eta_i\|_{v_j})$. Note that the $S$-Regulator $R_S$ is independent of the choice of the minor and the fundamental system of units $\eta_1,\dots,\eta_{s-1}$.

According to Bugeaud and Gy\H{o}ry \cite{Bugeaud:1996} we can choose a fundamental system of units that will be useful for our applications.

\begin{lemma}\label{lem:Dirichlet}
Put
\begin{gather*}
c_1=c_1(d,s)=\frac{2\left((s-1)!\right)^2}{(2d)^{s-1}}, \qquad c_2=c_2(d,s)=c_1 \left(\frac{d}{\lambda(d)}\right)^{s-2}\\
c_3=c_3(d,s)=\frac{c_1 d^{s-1}}{\lambda(d)},\qquad c_4=c_4(d,s)=\frac{(s-1)^s}{2(\lambda(d))^{s-2}}.
\end{gather*}
Then there exists a system of fundamental units $\eta_1,\dots,\eta_{s-1}$ such that
\begin{enumerate}
 \item $\prod_{i=1}^{s-1} h(\eta_i)\leq c_1 R_S$,
 \item $h(\eta_i)\leq c_2 R_S$ for all $i=1,\dots,s-1$.
 \item Let $E=(e_{ij})$ be the inverse of one of the $s-1\times s-1$ submatrices of $R$. Then we have $|e_{ij}|\leq c_3$.
 \item For every $\alpha\in \ord_S\setminus\{0\}$ and every integer $n\geq 1$ there exists $\epsilon\in U_S$ such that
 $$h(\alpha\epsilon^n)\leq \frac{\log N_S(\alpha)}{d}+nc_4R_S.$$
 \end{enumerate}
\end{lemma}

For a proof of statements (1)-(3) see Bugeaud and Gy\H{o}ry \cite[Lemma 1]{Bugeaud:1996}. Let us note that in the case that $S$ is the set of archimedian places a proof of statement~(4) can also be found in \cite[Lemma 2]{Bugeaud:1996}. However, this proof can be extended to any set of places as proofs of variants of statement (4) show (see e.g. \cite{Hajdu:1993} or \cite{Bugeaud:1998}).

Moreover, we note that Voutier's lower bound for the heigth \eqref{eq:Voutier-est} together with statement (1) of Lemma \ref{lem:Dirichlet} implies
\begin{equation}\label{eq:Reg-bound}
 R_S\geq \left(\frac{\lambda(d)}{d}\right)^{s-1} c_1^{-1}=\frac{1}{((s-1)!)^2}\cdot\left(\frac{16}{9(\log(3d))^3}\right)^{s-1}.
\end{equation}

The following lemma is similar to a statement due to Evertse et.al. \cite[Lemma~5]{Evertse:1988}.

\begin{lemma}\label{lem:S-normalized}
 Each $S$-equivalence class contains a triple $(a_1,a_2,a_3)$ with the following properties:
 \begin{enumerate}
  \item $a_1,a_2,a_3\in \ord_S\setminus\{0\}$,
  \item $N_S(a_1)\leq N_S(a_2)\leq N_S(a_3)$,
  \item $\prod_{v\not\in S}\max\{\|a_1\|_v,\|a_2\|_v,\|a_3\|_v\}\geq |D_K|^{-1/(2d)}$,
  \item $h(a_i):=\min_{x\in U_S}\{h(a_ix)\}\leq \frac{\log N_S(a_i)}{d}+c_4R_S$ for $i=1,2,3$.
 \end{enumerate}
\end{lemma}

\begin{proof}
 We follow the ideas of Evertse et.al. in \cite[Lemma 5]{Evertse:1988}. Therefore, assume that $(\alpha_1,\alpha_2,\alpha_3)\in K^3$. Let $\mathfrak d$ be the inverse ideal of the (fractional) ideal generated by $\alpha_1,\alpha_2,\alpha_3$. There exists (see e.g. \cite[Chapter 1, Lemma 6.2]{Neukirch:AZT}) a $\delta\in \mathfrak d$ such that
 $$|N_{K/\Q}(\delta)|\leq |D_K|^{1/2} N_{K/\Q}(\mathfrak d).$$
 Put $\alpha_i'=\delta \alpha_i$ for $i=1,2,3$, then we have $\alpha_i'\in \ord_K\setminus\{0\}$ and we get 
 \begin{multline*}
 N_{K/\Q}((\alpha_1',\alpha_2',\alpha_3'))=|N_{K/\Q}(\delta)|N_{K/\Q}((\alpha_1,\alpha_2,\alpha_3))\\
 \leq |D_K|^{1/2} N_{K/\Q}(\mathfrak d) N_{K/\Q}((\alpha_1,\alpha_2,\alpha_3))=|D_K|^{1/2}. 
 \end{multline*} 
 Let $S_\infty\subseteq S$ be the set of archimedian places contained in $S$. Then, we deduce due to the product formula that 
 \begin{align*}
 N_{K/\Q}((\alpha_1',\alpha_2',\alpha_3'))&=\left(\prod_{\nu\not\in S_\infty} \max\{\|\alpha_1'\|_v, \|\alpha_2'\|_v,\|\alpha_3'\|_v\}\right)^{-d}\\
 &\leq \left(\prod_{\nu\not\in S} \max\{\|\alpha_1'\|_v, \|\alpha_2'\|_v,\|\alpha_3'\|_v\}\right)^{-d}.
 \end{align*}
 Altogether we obtain point (3) of the lemma.
 
 By a simple reordering we may assume that $N_S(\alpha'_1)\leq N_S(\alpha'_2)\leq N_S(\alpha'_3)$. Finally, we can choose $\epsilon_i\in U_S$ for $i=1,2,3$ such that for $a_i=\epsilon_i\alpha'_i$ we have
 $$h(a_i)=\min_{x\in U_S}\{h(xa_i)\}$$
 and due to Lemma \ref{lem:Dirichlet} we have that $h(a_i)\leq \frac{\log N_S(a_i)}{d}+c_4R_S$ for $i=1,2,3$. Note that by construction $\alpha'_i\in \ord_K$ and therefore $a_i\in \ord_S$. Also note that
 $$\prod_{v\not\in S} \max\{\|\alpha_1'\|_v, \|\alpha_2'\|_v,\|\alpha_3'\|_v\}=\prod_{v\not\in S} \max\{\|a_1\|_v, \|a_2\|_v,\|a_3\|_v\}.$$
\end{proof}

A triple $(a_1,a_2,a_3)$ that satisfies the properties of Lemma \ref{lem:S-normalized} is called $S$-normal\-ized. That is every $S$-equivalence class contains at least one $S$-normalized triple. Note that the definition of $S$-normalized differs from the notation $S$-normalized used in \cite{Evertse:1988}.

\begin{lemma}
 Assume that $(a_1,a_2,a_3)\in (K^*)^3$, with $\nu(a_1,a_2,a_3)>0$ and at least one of $a_1/a_2,a_1/a_3,a_2/a_3$ is an $S$-unit. Then $(a_1,a_2,a_3)$ is equivalent to an $S$-normalized triple of the form $(1,1,c)$ with $c\in \ord_S$. 
\end{lemma}

\begin{proof}
 After a permutation we may assume that $a_2/a_1=u \in U_S$ that is $(a_1,a_2,a_3)$ is equivalent to $(1,u,a_3/a_1)$ and since $u\in U_S$ it is equivalent to 
 $(1,1,c')$ for some $c\in K^*$. Since $\nu(a_1,a_2,a_3)\geq 1$ we know that there are $x,y\in U_S$, with $x+y=c'$, hence $c'\in \ord_S$. Due to Lemma \ref{lem:Dirichlet} there is an $S$-unit $\epsilon$ such that $c=c'\epsilon$ satisfies 
 $$h(c)=\min_{x\in U_S}\{h(xc')\}\leq \frac{\log N_S(c)}d +c_4 R_S.$$
 With this choice $(1,1,c)$ is $S$-normalized.
\end{proof}

Let $S$ be a set of places and $(a_1,a_2,a_3)\in (K^*)^3$. Let us write 
$$A=\max\{1,h(a_1),h(a_2)\}.$$
If $S=S_\infty$ is exactly the set of archimedian primes, we set $P=1$. Otherwise we define
$$P=\max_{v\in S\setminus S_{\infty}}\{p\in \P \: :\: v|p\}.$$
For fixed  $(a_1,a_2,a_3)\in (K^*)^3$ let $(x_1,y_1),\dots,(x_n,y_n)$ be all solutions to \eqref{eq:SUeq}. We write
$$X=\max_{i=1,\dots,n}\{h(x_i),h(y_i)\}.$$
Now, we can state our main theorem. 

\begin{theorem}\label{th:number}
 Let $(a_1,a_2,a_3)\in (K^*)^3$ be an $S$-normalized triple. Assume that $(a_1,a_2,a_3)$ is not equivalent to $(1,1,c)$ for some $c\in \ord_S$.
 Let us define
 $$C=C(s,d,P)= 2.3\cdot 10^8 d (s+1)^{6.5}(\log(2(s+1)d))^2 P^d((s-1)!)^2\left(128e^2 d^2\right)^s.$$
 If $\nu(a_1,a_2,a_3)\geq \frac{|S|+1}2$, then we have
 $$\max\{X,h(a_3)\}\leq  3CR_SA \log\left(CR_SA\right).$$
  
 Suppose that $(1,1,c)$ with $c\in \ord_S$ is $S$-normalized such that $\nu(1,1,c)\geq |S|+1$. Then we have
 $$\max\{X,h(c)\}\leq 3CR_S \log\left(CR_S\right).$$
\end{theorem}

Let us note that the second statement of Theorem \ref{th:number} is, up to improved constants, covered by the result due to Evertse, et.al. \cite{Evertse:1988}. Also note that the two statements of Theorem \ref{th:number} coincide, if we count the two solutions $x+y=c$ and $y+x=c$ as only one solution.

We also want to note that using ineffective methods (see \cite[Theorem 1]{Evertse:1988} or \cite[Theorem 6.1.6]{Evertse-Gyoery:UnitEq}) it can be shown that there are at most finitely many $S$-equivalence classes such that the $S$-unit equation \eqref{eq:SUeq} has more than two solutions.

Using our strategy it is also possible to prove the following related theorem.

\begin{theorem}\label{th:unit-rep}
 Let $K$ be a number field of degree $d$, unit group $U_K$ of rank $s-1$.  Let $a,b\in \Z$ be fixed positive integers with $\gcd(a,b)=1$. Assume that $n\in \Z$ has a representation of the form 
 \begin{equation}\label{eq:unit-rat} ax+by=n\end{equation}
 with $x,y\in U_K$. Let us define $A=\max\{1,\log|a|,\log|b|\}$ and let us write
 $$\tilde C=\tilde C(s,d)= 6.2\cdot 10^{12} s^{5.5}d^6\log(ed^2)((s-1)!)^4 \left(54(\log(3d))^3\right)^{2s-1} .$$
 Then we have
 \begin{itemize}
  \item $\log |n|\leq 5\tilde CAR_K^2\log\left(\tilde CAR_K^2\right)$ or
  \item $a=b=1$ and $x=\sigma(y)$ for some $\sigma \in \Gal(K/\Q)$.
 \end{itemize}
\end{theorem}

Let us note that Theorem \ref{th:unit-rep} is closely related to a result due to Brindza and Gy\H{o}ry \cite{Brindza:1990}. Brindza and Gy\H{o}ry proved a version of Theorem \ref{th:unit-rep} without explicitly describing the possible solution to \eqref{eq:unit-rat} in case that $a=b=1$. Recently R. Visser and the author \cite{Visser:2026} used another approach to obtain the explicit description of the possible solution to \eqref{eq:unit-rat} in the case that $a=b=1$, but they did not consider the case that $ab>1$. Therefore, we give a complete proof of Theorem \ref{th:unit-rep} using the method of Visser and the author (see Section \ref{sec:units}).

Using ineffective methods it can be shown (see \cite[Theorem 1]{Brindza:1990}) that at most finitely many triples $(a,b,c)\in \Z^3$ exist with $\gcd(a,b)=1$ such that  the unit Equation \eqref{eq:unit-rat} has a solution $(x,y)$ for which $\Q(x,y)$ is not contained in a real quadratic field.

\section{Results on Diophantine approximation and S-unit equations}\label{sec:DioApprox}

One of the main tools in our proof are lower bounds for linear forms in (complex and $p$-adic) logarithms. We start with the following result due to Matveev \cite{Matveev:2000}.

\begin{lemma}\label{lem:matveev}
  Denote by $\alpha_1, \dots, \alpha_N$ algebraic numbers, not $0$ nor $1$.  Let $K=\Q(\alpha_1,\ldots,\alpha_N)$ and $D=[K:\Q]$ and denote by $b_1, \dots, b_N$ rational integers with $b_N\neq 0$. Furthermore, let $\kappa=1$ if $K$ is real and $\kappa=2$ otherwise. For all integers $j$ with $1\leq j\leq N$ choose
 $$
    A_j\geq h_M(\alpha_j):=\max\{D h(\alpha_j), |\log\alpha_j|,0.16\},
  $$
  and set
  $$
    E=\max ( \{1\} \cup \{|b_j| A_j /A_N\: :\: 1\leq j \leq N \} ).
  $$
  Assume that $\Lambda = b_1 \log \alpha_1 + \cdots + b_N \log \alpha_N\neq 0$. Then
  \begin{equation*}
    \log |\Lambda|
    \geq -C(N,\kappa)D^2 \Omega\log(eD)\log(eE)
  \end{equation*}
  with $\Omega=A_1\cdots A_N$ and
  $$C(N,\kappa)= \min\left\{\frac 1{\kappa}\left(\frac {eN}2\right)^\kappa 30^{N+3} N^{3.5}, 2^{6N+20}\right\}.$$
  \end{lemma}
  
  Here $\log \alpha$ can be any nonzero determination of the logarithm of $\alpha$. However, we will choose the principal value, that is we will assume that the imaginary part of $\log \alpha$ is contained in the interval $(-\pi,\pi]$.

In the case of $p$-adic logarithms we use a result due to Yu \cite{Yu:2007}.

\begin{lemma}\label{lem:Yu}
 Denote by $\alpha_1, \dots, \alpha_N$ algebraic numbers, not $0$ nor $1$. Let $K=\Q(\alpha_1,\ldots,\alpha_N)$ and $D=[K:\Q]$ and denote by $b_1, \dots, b_N$ rational integers. Let $\mathfrak p$ be a prime ideal of $\ord_K$ with $\mathfrak p|p$ and $p$ a prime number. Moreover we denote by $e_{\mathfrak p}$ the ramification index of $\mathfrak p$ and by $f_{\mathfrak p}$ the residue class degree of $\mathfrak p$. Assume that
$$\Lambda=\alpha_1^{b_1}\cdots \alpha_N^{b_N}-1$$
is nonzero. Then we have
\begin{equation}\label{eq:Yu}
\order_{\mathfrak p}(\Lambda)\leq C(N,D) H_1\cdots H_N \log B
\end{equation}
where
$$C(N,D)=(16eD)^{2(N+1)}N^{5/2}\log(2ND)\log(2D)e_{\mathfrak p}^N\frac{p^{f_{\mathfrak p}}}{(f_{\mathfrak p}\log p)^2}$$
and $B\geq 3$ is a number such that $|b_i|\leq B$ for $1\leq i\leq N$ and
$$ H_i\geq h_Y(\alpha_j):=\max\left\{h(\alpha_i),\frac{1}{16e^2d^2}\right\} \quad\quad \hbox{for $i=1,\dots,N$}.$$
\end{lemma}

The following lemma will be useful in applying the results of Matveev and Yu stated above. We will denote by $\zeta\in K$ a fixed primitive root of unity and we write $\omega=|\langle \zeta \rangle|$ for the number of roots of unity in $K$. For the next results we will introduce the notion $h^*(\alpha)$ for $\alpha\in K^*$, where $h^*(\alpha)=h(\alpha)$ if $\alpha$ is not a root of unity and $h^*(\alpha)=1$ otherwise. The following lemma is an application of the lower bounds for linear forms of logarithms and similar results can be found in \cite[Section 4.2]{Evertse-Gyoery:UnitEq}.

\begin{lemma}\label{lem:Fu-Units-lin-form}
 Let $\eta_1,\dots,\eta_{s-1}$ be a fundamental system of $S$-units of $K$ that satisfies the statements in Lemma \ref{lem:Dirichlet} and let $\alpha\in K^*$ be fixed. Let $b_0,b_1,\dots,b_{s-1}$ be integers with $B:=\max\left\{|b_1|,\dots,|b_{s-1}|,\frac{es\omega}2\right\}$ and $(b_1,\dots,b_{s-1})\neq (0,\dots,0)$. Then we have
 either $\alpha\zeta^{b_0}\eta_1^{b_1}\cdots \eta_{s-1}^{b_{s-1}}=1$ or
 \begin{equation}\label{eq:Linformbound}
  -\log \|\alpha\zeta^{b_0}\eta_1^{b_1}\cdots \eta_{s-1}^{b_{s-1}}-1\|_v\leq C_v  h^*(\alpha)R_S\log B,
 \end{equation}
where 
$$
C_v=3.8\cdot 10^7 d^2\log(3d) (s+1)^{5.5}((s-1)!)^2 (54(\log(3d))^3)^{s}
$$
in case that $v\in S_\infty$ and
$$C_v=1.6\cdot 10^4 d^2 (s+1)^{2.5}(\log(2(s+1)d))^2 P^d((s-1)!)^2\left(128 e^2 d^2\right)^s
$$
in case that $v\in S\setminus S_\infty$.
\end{lemma}

\begin{proof}
We assume that $\alpha\zeta^{b_0}\eta_1^{b_1}\cdots \eta_{s-1}^{b_{s-1}}\neq 1$ and start with the case that $v$ is archimedian. First, we show that if $K\neq \Q$ or $K=\Q$ and $\gamma>0$ we have
 \begin{equation}\label{eq:mod-height-Matv-ieq}
  h_M(\gamma)\leq h(\gamma)\left(\frac{\pi d}{\lambda(d)}+\frac 1d\right)\leq \frac{3.2 d}{\lambda(d)} h(\gamma).
 \end{equation}
 for $\gamma\in K^*$ not a root of unity.
 
 We note that 
 $$\frac{d h(\gamma)}{\lambda(d)}\cdot0.16\geq 0.16.$$
 Therefore \eqref{eq:mod-height-Matv-ieq} holds in case that $h_M(\gamma)=0.16$. We also note that 
 $$
 dh(\gamma)\geq h(\gamma)\left(\frac{\pi d}{\lambda(d)}+\frac 1d\right).
 $$
 Thus we may assume that $h_M(\gamma)=|\log \gamma|$ and we obtain
 \begin{multline*}
  |\log \gamma|\leq \log|\gamma|+\pi\leq \frac{h(\gamma)}d+\pi\leq h(\gamma)\left(\frac{\pi}{h(\gamma)}+\frac 1d\right)\leq h(\gamma)\left(\frac{\pi d}{\lambda(d)}+\frac 1d\right). 
 \end{multline*}
This proves the first inequality of \eqref{eq:mod-height-Matv-ieq}. For the second inequality note that $\lambda(d)=\frac{8}{9(\log 3d)^3}$. That is we have
$$
\frac{\pi d}{\lambda(d)}+\frac 1d\leq \frac{d}{\lambda(d)}\left(\pi +\frac{8}{9d^2 (\log 3d)^3}\right)<\frac{3.2 d}{\lambda(d)}.
$$
if $d\geq 2$. In the case that $K=\Q$ and $\gamma>0$ we have $h_M(\gamma)=dh(\alpha)=h(\alpha)$ and Inequality \eqref{eq:mod-height-Matv-ieq} holds trivially. 

Let us write $\Gamma=\alpha\zeta^{b_0}\eta_1^{b_1}\cdots \eta_{s-1}^{b_{s-1}}-1$ and $\Lambda=\log (\Gamma+1)$. If $|\Gamma|\geq 1/2$ the lemma obviously holds. In case that $|\Gamma|< 1/2$ we deduce that $|\Lambda|<2 |\Gamma|$.
Note that 
\begin{align*}
\Lambda&=\log \alpha+b_0\log \zeta+b_1\log \eta_1+\dots+b_{s-1}\log\eta_{s-1}-k i\pi\\
&=\log \alpha+\tilde b\log \zeta+b_1\log \eta_1+\dots+b_{s-1}\log\eta_{s-1},
\end{align*}
for suitable integers $k$ and $\tilde b=\frac{\omega}2 k+b_0$ such that the imaginary part of $\Lambda$ is contained in the interval $(-\pi,\pi]$. That is,
$$|\tilde b|\leq \frac{B(s-1)\omega+\omega}2+\frac{\omega}2\leq \frac{Bs\omega}2.$$

We apply Matveev's lower bound to $|\Lambda|$. Note that $h_M(\zeta)=\max\{\frac{2\pi}{\omega},0.16\}\leq \pi$, hence due to Lemma \ref{lem:Dirichlet} and Inequality \eqref{eq:mod-height-Matv-ieq} we have
\begin{align*}
\Omega&\leq h^*(\alpha)\pi \left(\frac{3.2 d}{\lambda(d)}\right)^s
c_1 R_S\\
&= h^*(\alpha)\pi \left(\frac{3.2 d}{\lambda(d)}\right)^s \frac{2\left((s-1)!\right)^2}{(2d)^{s-1}} R_S \\
&= 4d\pi((s-1)!)^2 \left(\frac{1.6}{\lambda(d)}\right)^{s} h^*(\alpha)R_S.
\end{align*}

Since we assume that $B\geq \frac{es\omega}2$ we have
$$e\max\{|\tilde b|,|b_1|,\dots,|b_{s-1}|\}\leq \frac {e^2s\omega}2B\leq B^2.$$
This yields
$$\log(eE)\leq \log \left(e\max\{|\tilde b|,|b_1|,\dots,|b_{s-1}|\}\right)\leq 2 \log B.$$
Therefore we obtain
\begin{multline*}
-\log |\Lambda| \leq \frac 1{\kappa}\left(\frac{e(s+1)}2\right)^\kappa 30^{(s+1)+3}(s+1)^{3.5}d^2 \\
\times 4d\pi((s-1)!)^2 \left(\frac{1.6}{\lambda(d)}\right)^{s} h^*(A)R_S \log (ed) 2\log B
\end{multline*}
Since $|\Lambda|^{\kappa/d}=\|\Lambda\|_v$ we obtain that
$$-\log (\|\Gamma-1\|_v)\leq- \log (2\|\Lambda\|_v)\leq -\log 2-\frac{\kappa}d \log|\Lambda|$$
and therefore we have
\begin{align*}
-\log (\|\Gamma-1\|_v)& \leq \frac{\kappa}d \cdot \frac 1{\kappa}\left(\frac{e(s+1)}2\right)^\kappa 30^{(s+1)+3}(s+1)^{3.5}d^2 \\
&\qquad \times 4d\pi((s-1)!)^2 \left(\frac{1.6}{\lambda(d)}\right)^{s} h^*(\alpha)R_S\log (ed) 2\log B+\log 2\\
&\leq 2e^2\pi 30^4 d^2\log(3d) (s+1)^{5.5}((s-1)!)^2 \left(\frac{48}{\lambda(d)}\right)^{s}\\
&\qquad\times h^*(\alpha)R_S\log B+\log 2\\
&\leq 3.8\cdot 10^7 d^2\log(3d) (s+1)^{5.5}((s-1)!)^2 (54d(\log(3d))^3)^{s}\\
&\qquad \times h^*(\alpha)R_S\log B 
\end{align*}
which proves the lemma for $v\in S_\infty$ in case that $K\neq \Q$. Note that in the case that $K=\Q$ we may choose $\eta_1,\dots,\eta_s$ to be positive and without loss of generality we can also choose $\zeta=1$ by changing the sign of $\alpha$ if necessary. In case that $\alpha>0$ the same arguments as above would lead to an even sharper bound for $-\log \|\Gamma -1\|_v$. In case that $\alpha<0$ we have $\Gamma<0$ and obtain $-\log\|\Gamma-1\|_v<0$ which also proves the lemma in this case. 

We consider the case that $v\in S$ is finite. In particular, we assume that $v$ corresponds to the prime $\mathfrak p|p$. Let us note that due to Voutier's lower bound \eqref{eq:Voutier-est} for heights we have $h_Y(\gamma)=h(\gamma)$ provided that $\gamma$ is not a root of unity. We also note that $h_Y(\zeta)\leq 1$.
 
 Due to Lemma \ref{lem:Dirichlet} we have
 $$
 h_Y(\alpha)h_Y(\zeta)h_Y(\eta_1)\cdots h_Y(\eta_{s-1})\leq  c_1 R_S h^*(\alpha)
 =\frac{2((s-1)!)^2}{(2d)^{s-1}} R_Sh^*(\alpha).
$$
 Hence, we get by Yu's lower bound (Lemma \ref{lem:Yu}), with $N=s+1$ and $D=d$ the inequality
 \begin{multline*}
 \order_{\mathfrak p}(\Gamma-1)\leq (16ed)^{2s+2}(s+1)^{2.5}\log(2(s+1)d)\log(2d)e_{\mathfrak p}^{s+1}\frac{p^{f_{\mathfrak p}}}{(f_{\mathfrak p}\log p)^2}\\
 \times \frac{2((s-1)!)^2}{(2d)^{s-1}}  R_Sh^*(\alpha)\log B.
 \end{multline*}
 Since we have $p\leq P$ and $e_{\mathfrak p},f_{\mathfrak p}\leq d$, we obtain
 \begin{multline*}
 \order_{\mathfrak p}(\Gamma-1)\leq 1024 e^2 d^2 (s+1)^{2.5}(\log(2(s+1)d))^2 \frac{p^d}{(\log p)^2}((s-1)!)^2\\
 \times \left(128 e^2 d^2\right)^s R_Sh^*(\alpha)\log B.
 \end{multline*}
 Finally, let us note that $\order_\mathfrak p(\alpha)=-\frac{\log\|\alpha\|_v}{\log p}$. We also note that $\frac{p^d}{\log p}\leq \frac{P^d}{\log 2}$, which yields the desired inequality.
 \end{proof}

 For a unified bound which holds for all places $v\in S$ we note that 
 \begin{equation}\label{eq:Cv-bound}
  C_v\leq 3.8\cdot 10^7 d^2 (s+1)^{5.5}(\log(2(s+1)d))^2 P^d((s-1)!)^2\left(128e^2 d^2\right)^s
 \end{equation}
 holds for all $v\in S$.
 
 Next, we prove the following useful lemma that relates the height of the $S$-unit $x$ to the maximal exponent that appears, if $x$ is written as a product of fundamental units.

 \begin{lemma}\label{lem:height-exp}
  Let $x\in U_S$ and let $\eta_1,\dots,\eta_{s-1}$ be a fundamental system of units satisfying the statements in Lemma \ref{lem:Dirichlet} and $\zeta$ be a primitive root of unity of $K$. Assume that $x=\zeta^{b_0}\eta_1^{b_1}\cdots \eta_{s-1}^{b_{s-1}}$ and write $B=\max\{b_1,b_2,\dots,b_{s-1}\}$. Then we have
  $$C_1B<h(x)<C_2BR_S,$$
  with
  $$
  C_1=\frac{2^s}{9d((s-1)!)^2(\log(3d))^3}
  $$
  and
  $$
  C_2= \frac{(s-1)((s-1)!)^2 }d \left(\frac{9(\log (3d))^3}{16} \right)^{s-2}.
  $$
 \end{lemma}

\begin{proof}
 First, we remind that the height satisfies
 \begin{itemize}
  \item $h(\alpha\beta)\leq h(\alpha)+h(\beta)$ and
  \item $h(\alpha^r)=|r|h(\alpha)$
 \end{itemize}
 for any nonzero algebraic numbers $\alpha$ and $\beta$ and rational $r$. By Lemma \ref{lem:Dirichlet} we have
 \begin{align*}
 h(x)&\leq \sum_{i=1}^{s-1}b_ih(\eta_i)\leq (s-1)B c_2 R_S\\
 &=(s-1) c_1\left(\frac{d}{\lambda(d)}\right)^{2-s} BR_S\\
 &\leq (s-1) \frac{2((s-1)!)^2}{(2d)^{s-1}}\left(\frac{9d(\log(3d))^3}{8}\right)^{s-2}BR_S\\
&\leq \frac{(s-1)((s-1)!)^2 }d \left(\frac{9(\log (3d))^3}{16} \right)^{s-2} BR_S
 \end{align*}
 which proves the upper bound for $h(x)$.

 On the other hand let $\lambda(x)=(\log\|x\|_{v_j})_{j=1,\dots,s}$, with $\{v_1,\dots,v_s\}=S$ and consider the $(s-1)\times s$-matrix $R=(\log\|\eta_i\|_{v_j})$. Then we have
 $$R\cdot \left(\begin{array}{c} b_1\\b_2\\ \vdots \\b_{s-1}\end{array}\right)= \lambda(x).$$
 Let $E$ be the inverse of $(\log\|\eta_i\|_{v_j})_{\substack{1\leq i \leq s-1\\1\leq j\leq s-1}}$, then we get
 $$\left(\begin{array}{c} b_1\\b_2\\ \vdots \\b_{s-1}\end{array}\right)=E \left(\begin{array}{c}\log\|x\|_{v_1}\\ \log\|x\|_{v_2}\\ \vdots \\ \log\|x\|_{v_{s-1}} \end{array}\right).$$
 Note that
 $$\sum_{j=1}^{s-1} |\log\|x\|_{v_j}|\leq 2dh(x).$$
 Therefore we obtain by Lemma \ref{lem:Dirichlet}
 \begin{align*}
  B&\leq c_32dh(x)\\
  &=\frac{c_1d^{s-1}}{\lambda(d)}dh(x)\\
  &= \frac{2((s-1)!)^2}{(2d)^{s-1}} d^{s-1} d \frac{9(\log(3d))^3}8 h(x)\\
  &=\frac{9d((s-1)!)^2(\log(3d))^3}{2^s}h(x)
  \end{align*}
 and we obtain a lower bound for $h(x)$.
\end{proof}

The following lemma will be the essential tool in proving Theorem \ref{th:number}.

\begin{lemma}\label{lem:4terms}
 Given $a_1,a_2\in K^*$, and assume that for some $x_1,y_1,x_2,y_2\in U_S$ we have
 $$a_1x_1+a_2y_1-a_1x_2-a_2y_2=0$$
 such that no subsum vanishes. For fixed $v\in S$ let us write
 $$x=\max\{\|x_1\|_v,\|x_2\|_v,\|y_1\|_v,\|y_2\|_v\}$$
 and 
 $$y={\min}^{(2)}\{\|x_1\|_v,\|x_2\|_v,\|y_1\|_v,\|y_2\|_v\},$$
 where $\min^{(2)}\{ M\}$ denotes the second smallest element of the finite set $M\subseteq \R$. Let us write
 $$H=\max\{h(x_1),h(x_2),h(y_1),h(y_2)\} \quad\text{and}\quad A=\max\{h^*(a_1),h^*(a_2)\}.$$
 Assume that $H>3/C_1$, then we have
 $$\log y>\log x-C_3(v)AR_S\log H, $$
 where $C_3(v)=3C_v.$
 \end{lemma}

 \begin{proof}
  Let us assume that $\sigma$ is a permutation on $\{1,2,3,4\}$ such that $x_1=u_{\sigma(1)}$, $x_2=u_{\sigma(2)}$, $y_1=u_{\sigma(3)}$ and $y_2=u_{\sigma(4)}$ and
  $$\|u_1\|_\nu\geq \|u_2\|_\nu\geq \|u_3\|_\nu\geq \|u_4\|_\nu.$$
  That is, we have the equation
  $$u_1\tilde a_1+u_2\tilde a_2+u_3\tilde a_3+u_4\tilde a_4=0$$
  with $\tilde a_i\in \{\pm a_1,\pm a_2\}$ and get the
  inequality
\begin{equation}\label{eq:fu-lem-ieq1}
\begin{split}
 -\log \left\|\frac{\tilde a_2 u_2}{\tilde a_1u_1}-1\right\|_v&\geq -\log \left\|\frac{\tilde a_3u_3+\tilde a_4 u_4}{\tilde a_1 u_1}\right\|_v\\
 &\geq -\log \frac{\|\tilde a_3\|_v\|\|u_3\|_v+\|\tilde a_4\|_v \|u_4\|_v}{\|\tilde a_1\|_v \|u_1\|_v}\\
 &\geq -\log\|u_3\|_v+\log\|u_1\|_v -\log\left(\frac{\|\tilde a_3\|_v+\|\tilde a_4\|_v}{\|\tilde a_1\|}\right)\\
 &\geq \log\|u_1\|_v-\log\|u_3\|_v-2dA-\log 2.
 \end{split}
\end{equation}
Indeed, we have
\begin{multline*}
\log\left(\frac{\|\tilde a_3\|_v+\|\tilde a_4\|_v}{\|\tilde a_1\|}\right)\\
\leq \log 2+ \max\{\log\|a_1\|_v,\log\|a_2\|_v\}-\min\{\log\|a_1\|_v,\log\|a_2\|_v\}\\
\leq \log 2 +2dA.
\end{multline*}

Next, we apply Lemma \ref{lem:Fu-Units-lin-form} to the left hand side of \eqref{eq:fu-lem-ieq1}, with $\alpha=\frac{\tilde a_2}{\tilde a_1}$ and $u_1/u_2=\zeta^{b_0}\eta_1^{b_1}\cdots\eta_{s-1}^{b_{s-1}}$. In particular, we have $h^*(\alpha)\leq 2A$. Let us denote by $\tilde B$ the maximum of the exponents of the fundamental units that appear in the decomposition of $u_1$ and $u_2$. Then we have due to Lemma \ref{lem:height-exp}
\begin{equation}\label{eq:B-est}
B=\max\{|b_1|,\dots,|b_s|\}\leq 2 \tilde B< \frac{2H}{C_1}.
\end{equation}

Let us assume for the moment that $B=\max\{|b_1|,\dots,|b_{s-1}|\}\leq \frac{es\omega}2$. Then we obtain by an application of Lemma \ref{lem:Fu-Units-lin-form} the inequality
$$
-\log \left\|\frac{\tilde a_2 u_2}{\tilde a_1u_1}-1\right\|_v\leq C_v2A(\log (es\omega/2)) R_S<2C_vAR_S\log(3sd)
$$
and together with \eqref{eq:fu-lem-ieq1} we get 
\begin{align*}
\log\|u_3\|_v &\geq \log\|u_1\|_v-2C_vAR_S\log(3sd)-2dA-\log 2\\
&> \log\|u_1\|_v-3C_vAR_s\log(3sd).
\end{align*}
Since we assume that $H\geq 3/C_1$ and since
$$\frac{3}{C_1}=\frac{27d((s-1)!)^2(\log(3d))^3}{2^s} \geq 3sd$$
we obtain the statement of the lemma in the case that $\max\{|b_1|,\dots,|b_{s-1}|\}\leq \frac{es\omega}2$.

Therefore we may assume that $\max\{|b_1|,\dots,|b_{s-1}|\}\geq \frac{es\omega}2$. By an application of Lemma \ref{lem:Fu-Units-lin-form} together with Inequality \eqref{eq:B-est} we get
\begin{equation}\label{eq:fu-lem-ieq2}
-\log \left\|\frac{\tilde a_2 u_2}{\tilde a_1u_1}-1\right\|_v\leq C_v2A(\log B) R_S<2C_vAR_S\log(2H/C_1)
\end{equation}
Combining the inequalities \eqref{eq:fu-lem-ieq1} and \eqref{eq:fu-lem-ieq2} yields
\begin{align*}
\log\|u_3\|_v &\geq \log\|u_1\|_v-2C_vAR_S\log(H/C_1)-2dA-\log 2\\
&> \log\|u_1\|_v-3AR_sC_v\log H
\end{align*}
provided that
\begin{equation}\label{eq:H-bound-I}
\log(H)> \log (2/C_1)+\frac{2dA+\log 2}{2AR_S C_v}.
\end{equation}
Let us note that by \eqref{eq:Reg-bound} it is easy to show that $R_SC_v\geq 1.6\cdot 10^4 d^2s^2$ for all $v\in S$ and therefore
$$\frac{2dA+\log 2}{2AR_S C_v}\leq \log(3/2).$$
Hence \eqref{eq:H-bound-I} is satisfied, if we we assume that $H>3/C_1$.
 \end{proof}
 
 For the proof of Theorem \ref{th:unit-rep} we need a slight variation of Lemma \ref{lem:4terms}. For this we fix some notation first. Let $\{\sigma_1,\dots,\sigma_d\}= \Hom_\Q(K,\C)$ be the set of all embedding from $K$ into $\C$ and let $K^{(i)}=\sigma_i(K)$ be the fields conjugate to $K$ embedded into $\C$. If $x\in K$ we write $x^{(i)}$ for $\sigma_i(x)$. We denote by $s-1$ the unit rank of the number field $K$, that is, $s$ is the number of archimedian places of $K$. Then we have the following lemma.
 
 \begin{lemma}\label{lem:4terms-conj}
 Given non zero integers $a_1,a_2$ with $\gcd(a_1,a_2)=1$ and $x,y\in U_K$ with
 $$a_1x^{(i)}+a_2y^{(i)}-a_1x^{(j)}-a_2y^{(j)}=0$$
 such that no subsum vanishes. For fixed indices $i,j\in \{1,\dots,d\}$, with $i\neq j$ we write
 $$X=\max\{|x^{(i)}|,|y^{(i)}|,|x^{(j)}|,|y^{(j)}|\}$$
 and 
 $$Y={\min}^{(2)}\{|x^{(i)}|,|y^{(i)}|,|x^{(j)}|,|y^{(j)}|\}.$$
 Let us write
 $$H=\max\{h(x),h(y)\} \quad\text{and}\quad A=\max\{1,\log|a_1|,\log |a_2|\}.$$
 Assume that $H>\max\{4esdC_2R_K,2/C_1\}$, then we have
 $$\log Y>\log X-\tilde C_3AR_K^2\log H, $$
 where 
 $$\tilde C_3=6.2\cdot 10^{12} s^{5.5}d^5\log(ed^2)((s-1)!)^4 \left(54(\log(3d))^3\right)^{2s-1}.$$
 \end{lemma}

 \begin{proof}
 We follow the arguments given in the proof of \cite[Lemma 6]{Visser:2026}. Without loss of generality we may distinguish between the following three cases
 \begin{description}
  \item[Case A] We have $|x^{(i)}|,|y^{(i)}|\geq |x^{(j)}|\geq |y^{(j)}|$.
  \item[Case B] We have $|y^{(i)}|,|y^{(j)}|\geq |x^{(j)}|\geq |x^{(i)}|$.
  \item[Case C] We have $|x^{(i)}|,|y^{(j)}|\geq |x^{(j)}|\geq |y^{(i)}|$.
 \end{description}
Note that all other possible cases can be obtained by these three cases by exchanging the roles of $x$ and $y$ and/or exchanging the roles of $i$ and $j$. We will discuss Case C in detail since the proof for the other cases is similar and the other cases will yield smaller bounds for $\tilde C_3$.

Let $\zeta$ be a primitive root of unity and $\eta_1,\dots,\eta_{s-1}$ a fundamental system of units of $K$. Then we have
\begin{align*}
x^{(i)}&=\left(\zeta^{(i)}\right)^{t_0}\left(\eta_1^{(i)}\right)^{t_1}\cdots \left(\eta_{s-1}^{(i)}\right)^{t_{s-1}}\\
x^{(j)}&=\left(\zeta^{(j)}\right)^{t_0}\left(\eta_1^{(j)}\right)^{\xi_1}\cdots \left(\eta_{s-1}^{(j)}\right)^{t_{s-1}}\\
y^{(i)}&=\left(\zeta^{(i)}\right)^{u_0}\left(\eta_1^{(i)}\right)^{u_1}\cdots \left(\eta_{s-1}^{(i)}\right)^{u_{s-1}}\\
y^{(j)}&=\left(\zeta^{(j)}\right)^{u_0}\left(\eta_1^{(j)}\right)^{u_1}\cdots \left(\eta_{s-1}^{(j)}\right)^{u_{s-1}}.
\end{align*}
Let us assume for the moment that $\left|a_2y^{(j)}\right|\geq \left|a_1x^{(i)}\right|$. We consider the inequality
$$\left|\frac{a_1x^{(i)}}{a_2y^{(j)}}-1\right|=\left|\frac{a_1x^{(j)}+a_2y^{(i)}}{a_2y^{(j)}}\right|\leq \frac{2|a_1x^{(j)}|}{|a_2y^{(j)}|}.$$
Let us write $\Gamma=\frac{a_1x^{(i)}}{a_2y^{(j)}}$. If $|\Gamma-1|\geq 1/2$, then the statement of the lemma holds. Thus we may assume that $|\Gamma-1|<1/2$. But this implies that
\begin{equation}\label{eq:log-ieq}
|\log \Gamma|<2 \frac{2|a_1x^{(j)}|}{|a_2y^{(j)}|}.
\end{equation}
Let $\Lambda=\log \Gamma$, then we have
\begin{multline*}
\Lambda=\log \alpha+t_0\log \zeta^{(i)}+t_1\log \eta_1^{(i)}+\dots +t_{s-1}\log \eta_{s-1}^{(i)}\\- u_0\log \zeta^{(j)}-u_1\log \eta_1^{(j)}-\dots -u_{s-1}\log \eta_{s-1}^{(j)}- k i\pi,
\end{multline*}
where $\alpha=a_1/a_2$ and $k$ is some integer such that the imaginary part of $\Lambda$ is contained in the interval $(-\pi,\pi]$. That is, we have
\begin{multline*}
\Lambda=\log \alpha+\tilde k\log\zeta^{(i)}+t_1\log \eta_1^{(i)}+\dots +t_{s-1}\log \eta_{s-1}^{(i)}\\-u_1\log \eta_1^{(j)}-\dots -u_{s-1}\log \eta_{s-1}^{(j)},
\end{multline*}
with $\tilde k$ some integer with $|\tilde k|< (2(s-1)B+3)\frac{\omega}2$. We apply Matveev's Theorem to $\Lambda$. We note that $\Lambda$ is a linear form in $2s$ logarithms and all algebraic numbers come from the compositum $K^{(i)}K^{(j)}$ of degree $\leq d(d-1)<d^2$. Similar as in the proof of Lemma \ref{lem:Fu-Units-lin-form} we deduce that
\begin{equation}\label{eq:omega-est}
 \Omega\leq A\pi \left(\frac{3.2 d}{\lambda(d)}\right)^{2s-1} c_1^2 R_K^2
 \end{equation}
We also have 
\begin{equation}\label{eq:E-est}
eE=e\max\{|\tilde k|,|u_1|,\dots,|u_s|,|t_1|,\dots,|t_s|\}\leq e(2sB+3)\frac{\omega}2\leq B^2
\end{equation}
provided that $B> 4esd$. Indeed, we have $s<d$ and $\omega \leq 2d$, that is,
$$e(2sB+3)\frac{\omega}2<2esdB+3Bd<4esdB<B^2$$
unless $B\leq 4esd$.

Let us assume for the moment that $B\leq 4esd$. Then we have due to Lemma \ref{lem:height-exp} the upper bound $H\leq C_2BR_K\leq 4esdC_2R_K$.

Now, let us consider the case that $B\geq 4esd$. Hence, we obtain by an application of Matveev's lower bound for linear forms in logarithms (Lemma \ref{lem:matveev}) with $N=2s$, $D=d^2$ and the estimates \eqref{eq:omega-est} and \eqref{eq:E-est} the inequality
\begin{align*}
-|\log \Lambda|<&\frac 1{\kappa}\left(\frac {e2s}2\right)^\kappa 30^{2s+3} (2s)^{3.5}d^4 A\pi \left(\frac{3.2 d}{\lambda(d)}\right)^{2s-1} c_1^2 R_K^2\log(ed^2)2\log(B)\\
\leq &e^2 \pi 2^{3.5}s^{5.5} 30^{2s+5}d^4 \left(\frac{3.2 d}{\lambda(d)}\right)^{2s-1} c_1^2 \log (ed^2) AR_K^2\log B\\
=&e^22^{7.5}30^6\pi s^{5.5}d^5\log(ed^2)((s-1)!)^4\left(\frac{48}{\lambda(d)}\right)^{2s-1}AR_K^2\log B\\
<&\stackrel{:=\tilde C_4}{\overbrace{3.1\cdot 10^{12} s^{5.5}d^5\log(ed^2)((s-1)!)^4 \left(54(\log(3d))^3\right)^{2s-1}}}AR_K^2\log B
\end{align*}
On the other hand we have
$$|\log \Lambda|<\log\left(\frac{4|a_1x^{(j)}|}{|a_2y^{(j)}|}\right)<\log 4+\log \alpha +\log y-\log x$$
Since $B<H/C_1$ and $\log \alpha<A$ we obtain
$$\log y>\log x-\log \alpha-\tilde C_4 A R_K^2 \log (H/C_1)>\log x-2\tilde C_4 A R_K^2\log H$$
provided that
\begin{equation}\label{eq:H-bound-II}
\log H>\log(1/C_1)+\frac{2dA+\log 2}{\tilde C_4 AR_K^2}.
\end{equation}
Due to Friedman \cite{Friedman:1989} we know that $R_K>0.2$.
Hence $\tilde C_4R_K^2>10^{10} s^5d^5$ and
$$\frac{2dA+\log 2}{\tilde C_4 AR_K^2}<\log 2. $$
Therefore, Inequality \eqref{eq:H-bound-II} certainly holds, if we assume that $H>2/C_1$. That is the statement of the lemma holds with $\tilde C_3=2\tilde C_4$ in Case C.

In case that $|a_2y^{(j)}|< |a_1x^{(i)}|$ we obtain instead of $\Lambda$ the linear form $-\Lambda$ and therefore we obtain the same lower bound for $\log y$.

Let us have a quick look at the Cases A and B and let us start with Case A.
In this case we obtain by following the arguments for Case C the linear form
\begin{multline*}
\Lambda'=\log \alpha+(t_0-u_0)\log \zeta^{(i)}
+(t_1-u_1)\log \eta_1^{(i)}+\\
\dots +(t_{s-1}-u_{s-1})\log \eta_{s-1}^{(i)}-\tilde k i\pi
\end{multline*}
instead of $\Lambda$. Note that we have a linear form in $s+1$ logarithms instead of $2s$. Also note that although the coefficients of the $\log \eta_k^{(i)}$ may be twice as large as those in $\Lambda$ we only have half of the logarithms. That is, the corresponding constant $\tilde C_4'$ is smaller than $\tilde C_4$ and we obtain a sharper result in this case.

Finally, let us consider Case B. Following the arguments for Case C we obtain in this case the linear form
$$\Lambda''= t_0\log \left(\frac{\zeta^{(i)}}{\zeta^{(j)}}\right)+t_1\log\left(\frac{\eta_1^{(i)}}{\eta_1^{(j)}}\right)+\dots +t_{s-1}\log\left(\frac{\eta_{s-1}^{(i)}}{\eta_{s-1}^{(j)}}\right)-\tilde k i\pi$$
instead of $\Lambda$. Note that we have a linear form in $s+1$ logarithms instead of $2s$. Moreover, we have $h\left(\eta_k^{(i)}/\eta_k^{(j)}\right)\leq 2h(\eta_k)$ for $k=1,\dots,s-1$. Therefore, the corresponding constant $\tilde C_4''$ will be smaller than $\tilde C_4$. That is, we obtain again a sharper result in this case.
 \end{proof}

Finally, we state a useful lemma due to Peth\H{o} and de Weger \cite{Pethoe:1986}:

\begin{lemma}\label{lem:pdw}
 Let $u,v \geq 0, h \geq 1$ and $x \in \R$ be the largest solution of $x=u+v(\log{x})^h$. Then
$$
x<\max\{2^h(u^{1/h}+v^{1/h}\log(h^hv))^h, 2^h(u^{1/h}+2e^2)^h\}.
$$
\end{lemma}

For a proof of Lemma \ref{lem:pdw} we refer to \cite[Appendix B]{Smart:DiGl}.

\section{A bound for the number of solutions}\label{sec:number}

This section is devoted to the proof of Theorem \ref{th:number}. Therefore let us assume that we have $n$ solutions $(x_1,y_1),\dots,(x_n,y_n)\in U_S^2$ to \eqref{eq:SUeq}. In case that $(a_1,a_2,a_3)=(1,1,c)$ we additionally assume that for no indices $i\neq j$ we have $x_i=y_j$ and $x_j=y_i$. That is the two solutions $x+y=c$ and $y+x=c$ are counted only once.

We consider the following homogeneous system of $S$-unit equations
\begin{equation}\label{eq:DioEq-Sys}
 a_1 X_i+a_2 Y_i-a_1X_j-a_2Y_j=0 \qquad 1\leq i<j\leq n.
\end{equation}
Then $(x_1,y_1,\dots,x_n,y_n)\in U_S^{2n}$ is a solution to \eqref{eq:DioEq-Sys} such that no subsum vanishes.
Let us write 
$$X=\max_{i=1,\dots,n}\{h(x_i),h(y_i)\}$$
and for $v\in S$ we write 
$$X_v=\max_{i=1,\dots,n}\{\|x_i\|_v,\|y_i\|_v\}.$$
We also continue to write $A=\max\{h^*(a_1),h^*(a_2)\}$. Then an immediate consequence of Lemma \ref{lem:4terms} is the following lemma.

\begin{lemma}\label{lem:one-excep}
We have either 
$$X\leq 3/C_1$$
or for all $v\in S$ there is at most one $x$ out of $\{x_1,\dots,x_n,y_1,\dots,y_n\}$ that satisfies
$$\log \|x\|_v< \log X_v-2C_3(v)AR_S\log X.$$
 \end{lemma}
 
 \begin{proof}
  Let us assume that $X>3/C_1$ and without loss of generality we may assume that $\|x_1\|_v=X_v$. Let us assume for the moment that $\|y_1\|_v \leq \log X_v-C_3(v)AR_S\log X$. Now for any $j=2,\dots,n$ we consider the equation
  $$a_1x_1+a_2y_1=a_3=a_1x_j+a_2y_j$$
  and due to Lemma \ref{lem:4terms} we deduce that 
  $$\log\|x_j\|_v,\log\|y_j\|_v\geq \log X_v-C_3(v)AR_S\log X$$
  for all $j=2,\dots,n$. This proves a slightly sharper result, than stated in the lemma.
  
  Without loss of generality we may assume that $\|y_n\|_v=\min_{i}\{\|x_i\|_v,\|y_i\|_v\}$ with $n\neq 1$. We consider the equation
  $$a_1x_1+a_2y_1=a_3=a_1x_n+a_2y_n$$
  and due to Lemma \ref{lem:4terms} we deduce that
   $$\log\|x_n\|_v\geq \log X_v-C_3(v)AR_S\log X.$$
   Now, we consider the equations
  $$a_1x_j+a_2y_j=a_3=a_1x_n+a_2y_n$$
  for all $j=1,\dots, n-1$ and deduce that
  \begin{align*}
  \log\|x_j\|_v,\log\|y_j\|_v&\geq\log \|x_n\|_v-C_3(v)AR_S\log X\\
  &\geq \log X_v-2C_3(v)AR_S\log X
  \end{align*}
  holds for all $j=2,\dots,n$ which proves the lemma.
  \end{proof}

 If $x$ is an $S$-unit, then with $(x_1,y_1,\dots,x_n,y_n)$ also $(xx_1,xy_1,\dots,xx_n,xy_n)$ is a solution to \eqref{eq:DioEq-Sys}. That is we may assume that $x_1=1$ by choosing $x=x_1^{-1}$. Next, we prove a result that is also of independent interest.

\begin{proposition}\label{prop:sol-bound}
  Assume that $(x_1,y_1,\dots,x_n,y_n)\in U_S^{2n}$ is a solution to \eqref{eq:DioEq-Sys} with $x_1=1$ such that no subsum vanishes. If $|S|\leq 2n-1$, then we have
 $$
  X=\max_{i=1,\dots,n}\{h(x_i),h(y_i)\}< 2C_4 R_SA\log(C_4R_S A),$$
   with
  $$C_4=\frac{2}d \sum_{v\in S} C_3(v).$$
\end{proposition}

\begin{proof}
It is easy to see that $2C_4 R_SA\log(C_4R_S A)>3/C_1$. Therefore we may assume that $X\geq 3/C_1$. By the pigeonhole principle combined with Lemma \ref{lem:one-excep} there is one $x\in \{x_1,\dots,x_n,y_1,\dots,y_n\}$ such that
 $$\log \|x\|_v\geq \log X_v-2C_3(v)AR_S\log X$$
 for all $v\in S$. Since by the product formula we have 
 \begin{align*}
 0&=\sum_{v\in S} \log \|x\|_v\\
 &\geq \sum_{v\in S}\log X_v -\sum_{v\in S}2C_3(v)AR_S\log X\\
 &=\sum_{v\in S}\log X_v -dC_4AR_S\log X.
 \end{align*}
 Now, let $z\in \{x_1,\dots,x_n,y_1,\dots,y_n\}$ be such that $h(z)=X$. Since we assume that $x_1=1$, we have $\log X_v\geq 0$ for all $v\in S$. Therefore, we have
 $$dC_4AR_S\log X\geq \sum_{v\in S}\log X_v\geq \sum_{v\in S} \max\{\log\|z\|_v,0\}= dh(z)=dX.$$
 Due to Lemma \ref{lem:pdw} we deduce that
 $$X<2C_4 R_SA\log(C_4R_S A).$$
\end{proof}

Let us assume that we have $n\geq \frac{|S|+1}2$ solutions $(x_1,y_1),\dots,(x_n,y_n)\in U_S^2$ to $S$-unit Equation \eqref{eq:SUeq}. In case that $(a_1,a_2,a_3)=(1,1,c)$ we additionally assume that for no indices $i\neq j$ we have $x_i=y_j$ and $x_j=y_i$. That is $(x_1,y_1,\dots,x_n,y_n)\in U_S^{2n}$ is a solution to \eqref{eq:DioEq-Sys} such that no subsum vanishes. Then we deduce by an application of Proposition \ref{prop:sol-bound} that $\mathbf x=(1,x_1^{-1}y_1,\dots,x_1^{-1}x_n,x_1^{-1}y_n)$ is a solution to  \eqref{eq:DioEq-Sys} with
$$X=\max\left\{h(x_1^{-1}y_1),\cdots,h(x_1^{-1}x_n),h(x_1^{-1}y_n) \right\}\leq 2C_4 R_SA\log(C_4R_S A).$$
That is, we have
$$a_1x_1^{-1}x_1+a_2x_1^{-1}y_1=a_3x_1^{-1}$$
and therefore
\begin{align*}
h(a_3)&\leq h(a_3 x_1^{-1})\\
&=h(a_1x_1^{-1}x_1+a_2x_1^{-1}y_1)\\
&\leq h(a_1)+h(a_2)+h(x_1^{-1}y_1)+\log 2\\
&\leq 2A+\log 2+2C_4 R_SA\log(C_4R_S A)\\
&<3C_4 R_SA\log(C_4R_S A).
\end{align*}
which proves Theorem \ref{th:number} up to the explicit computation of $C=C_4$.

Since \eqref{eq:Cv-bound} we can choose
\begin{multline*}
C_4= \frac{2}d \sum_{v\in S} C_3(v)\leq \frac{6}d \sum_{v\in S}C_v\\
\leq 2.3\cdot 10^8 d (s+1)^{6.5}(\log(2(s+1)d))^2 P^d((s-1)!)^2\left(128e^2 d^2\right)^s.
\end{multline*}

\section{Integers that are sums of two units}\label{sec:units}

Let us assume that $a,b\in \Z$ with $\gcd(a,b)=1$. Furthermore, assume that
$\mathbf x=(x,y)\in U_K$ is a solution to
\begin{equation}\label{eq:sum-of-units}
 ax+by=n
\end{equation}
Let us write $M=\Q(x)$ and let $\ell=[M:\Q]$. Since $y=\frac{n-ax}b$ we have $M=\Q(x)=\Q(y)$. Let $\{\tilde \sigma_1,\dots,\tilde \sigma_\ell\}=\Hom_{\Q}(M,\C)$ be the set of all embeddings of $M$ into $\C$ and let $\sigma_1,\dots,\sigma_\ell$ be fixed embeddings of $K$ into $\C$ with $\sigma_i|_M=\tilde \sigma_i$ for $i=1,\dots,\ell$. Let us write $\alpha^{(i)}=\sigma_i(\alpha)$ for $\alpha\in K$. Then we have
\begin{equation}\label{eq:unit-eq-conj}
ax^{(i)}+by^{(i)}-ax^{(j)}-by^{(j)}=0,\qquad i\neq j
\end{equation}
and no subsum vanishes provided that
\begin{itemize}
 \item $a\neq b$ or
 \item $a=b$ and $\sigma(x)\neq y$ for all $\sigma \in \Gal(K/\Q)$.
\end{itemize}
Indeed, since $x,y$ are units $a\sigma_i(x)= b\sigma_j(y)$ can only happen if $a=b=1$. In that case we have $\sigma_i(x)=\sigma_j(y)$, which can only happen if $\sigma(x)=y$ for some $\sigma\in \Gal(K/\Q)$.

Let us assume that $ax+by=n$. Assume for the moment that
$$\max\{h(x),h(y)\}\leq \max\{4esdC_2R_K,2/C_1\}. $$
Then we have
\begin{align*}
\log|n|=h(n)&\leq h(ax+by)\\
&\leq 2A+\log 2 +2 \max\{4esdC_2R_K,2/C_1\}\\
&\leq 5\tilde CAR_K^2\log\left(\tilde CAR_K^2\right)
\end{align*}
which proves Theorem \ref{th:unit-rep} in this case.

Let us assume that
$$\max\{h(x),h(y)\}> \max\{4esdC_2R_K,2/C_1\}. $$
Also in view of Theorem \ref{th:unit-rep} we also may assume that $|n|\geq a+b$. Let $(x,y)\in U_K^2$ be a solution to $ax+by=n$.
Then $(x,y)$ is also solution to \eqref{eq:unit-eq-conj}. Let us assume that no subsum of \eqref{eq:unit-eq-conj} vanishes, i.e. one of the two statements listed above holds.
 Let us write
 $$
 X^{(i)}=\max\left\{\left|x^{(i)}\right|,\left|y^{(i)}\right|\right\}.
 $$
 Since $|n|\geq a+b$ we have $X^{(i)}\geq 1$ for all $1\leq i\leq \ell$.
 Let us write $X=\max\{h(x),h(y)\}$. Further, let us assume that
 $$
 \left|y^{(\ell)}\right|=\min_{i=1,\dots,\ell}\left\{\left|x^{(i)}\right|,\left|y^{(i)}\right|\right\}.
 $$
 Note that by exchanging the roles of $a$ and $b$ and $x$ and $y$ respectively we can assume without loss of generality that the minimum is attained for one of the conjugates of $y$.
 
 With these notations and assumptions we obtain from Lemma \ref{lem:4terms-conj} the inequality
 $$\log\left|x^{(i)}\right|\geq X^{(i)}-\tilde C_3AR_K^2\log X$$
 for all $i=1,\dots,\ell$ and therefore
 $$
 0=\sum_{i=1}^{\ell}\log\left|x^{(i)}\right|\geq \sum_{i=1}^{\ell} X^{(i)}-\ell\tilde C_3A\log X\geq \sum_{i=1}^\ell X^{(i)}-d\tilde C_3A\log X.
 $$
Let $z\in \{x,y\}$ be such that $h(z)=\max\{h(x),h(y)\}$. Then we obtain the inequality
 $$d\tilde C_3A\log X\geq \sum_{i=1}^\ell X^{(i)}\geq \sum_{i=1}^\ell \max\left\{\log \left|z^{(i)}\right|,0\right\}=\ell h(z)\geq X.$$
 Therefore we have due to Lemma \ref{lem:pdw} the inequality
 $$X\leq 2d\tilde C_3 A R_K^2\log (d\tilde C_3 A R_K^2).$$
 Hence,
 \begin{align*}
 \log |n|=h(n)=h(ax+by)&\leq h(a)+h(b)+2X+\log2\\
 &\leq 5d\tilde C_3 A R_K^2 \log (d\tilde C_3 A R_K^2).
 \end{align*}
 That is, $|n|$ is bounded from above or $a=b=1$ and $n=x+y$ with $y=\sigma(x)$ for some $\sigma \in \Gal(K/\Q)$ which proves Theorem \ref{th:unit-rep} with
 $$\tilde C= d\tilde C_3=6.2\cdot 10^{12} s^{5.5}d^6\log(ed^2)((s-1)!)^4 \left(54(\log(3d))^3\right)^{2s-1}. $$

\section*{Acknowlegement}
We gratefully acknowledge Kálmán Győry for his valuable suggestions and helpful comments.

%\bibliographystyle{abbrv}
%\bibliography{units}

\end{document}